\documentclass{amsart}

\usepackage{amsmath}
\usepackage{amsthm}
\usepackage{amsbsy}
\usepackage{amssymb}
\usepackage{mathtools}
\usepackage{comment}
\usepackage{enumerate}
\usepackage[normalem]{ulem}
\usepackage[T1]{fontenc}
\usepackage{tikz}
\usetikzlibrary{positioning,patterns,calc}
\usepackage{mathrsfs}
\DeclareMathAlphabet{\mathpzc}{OT1}{pzc}{m}{it}
\usepackage[utf8]{inputenc}

\newcommand{\real}{\mathbb{R}} 
\newcommand{\complex}{\mathbb{C}} 

\newcommand{\sffM}{II} 
\newcommand{\scalar}{R} 
\newcommand{\Index}{\textup{Index}} 
\newcommand{\connM}{\overline{\nabla}} 
\newcommand{\connS}{\nabla^{\Sigma}} 
\newcommand{\tr}{\textup{tr}} 
\newcommand{\divergent}{\textup{div}} 
\newcommand{\dimension}{\textup{dim}} 
\newcommand{\harmonicvf}{\mathcal{H}^1} 
\newcommand{\harmonicvfb}{\mathcal{H}^1_T} 

\newtheorem{theorem}{Theorem}[section]
\newtheorem{lemma}[theorem]{Lemma}
\newtheorem{remark}[theorem]{Remark}
\newtheorem{proposition}[theorem]{Proposition}
\newtheorem{corollary}[theorem]{Corollary}

\newtheorem{definition}[theorem]{Definition}
\newtheorem*{theorem*}{Theorem}
\newtheorem{innercustomthm}{Theorem}
\newenvironment{maintheorem}[1]
  {\renewcommand\theinnercustomthm{#1}\innercustomthm}
  {\endinnercustomthm}
\newtheorem*{corollary*}{Corollary}
\newtheorem{innercustomcor}{Corollary}
\newenvironment{maincorollary}[1]
  {\renewcommand\theinnercustomcor{#1}\innercustomcor}
  {\endinnercustomthm}

\makeatletter
\@addtoreset{claim}{theorem}
\makeatother

\def\bpf{\begin{proof}}
\def\epf{\end{proof}}
\def\be{\begin{equation}}
\def\ee{\end{equation}}
\def\bea{\begin{eqnarray}}
\def\eea{\end{eqnarray}}
\def\bt{\begin{theorem}}
\def\et{\end{theorem}}
\def\bl{\begin{lemma}}
\def\el{\end{lemma}}
\def\br{\begin{remark}}
\def\er{\end{remark}}
\def\bc{\begin{corollary}}
\def\ec{\end{corollary}}
\def\bd{\begin{definition}}
\def\ed{\end{definition}}
\def\bp{\begin{proposition}}
\def\ep{\end{proposition}}

\usepackage{etoolbox}
\makeatletter
\patchcmd{\@maketitle}
  {\ifx\@empty\@dedicatory}
  {\ifx\@empty\@date \else {\vskip3ex \centering\footnotesize\@date\par\vskip1ex}\fi
   \ifx\@empty\@dedicatory}
  {}{}
\patchcmd{\@adminfootnotes}
  {\ifx\@empty\@date\else \@footnotetext{\@setdate}\fi}
  {}{}{}
\makeatother

\title[Index estimates for CMC surfaces in $3$-manifolds]{Index estimates for surfaces with constant mean curvature in $3$-dimensional manifolds}
\author{Nicolau S. Aiex and Han Hong}
\date{\today}
\address{University of British Columbia, Department of Mathematics, Vancouver BC V6T 1Z2, Canada}
\email{nsarquis@math.ubc.ca}
\address{University of British Columbia, Department of Mathematics, Vancouver BC V6T 1Z2, Canada}
\email{honghan@math.ubc.ca}

\begin{document}

\begin{abstract}
We prove index estimates for closed and free boundary CMC surfaces in certain $3$-dimensional submanifolds of some Euclidean space.
When the mean curvature is large enough we are able to prove that the index of a CMC surface in an arbitrary $3$-manifold is bounded below by a linear function of its genus.
\end{abstract}

\maketitle

\section{Introduction}

A closed hypersurface of constant mean curvature (CMC) may be variationally characterized a critical point of the area functional under variations that preserve enclosed volume.
In a similar way, a free boundary constant mean curvature (free boundary CMC) hypersurface is an extremal solution of the same problem where, in addition, the boundary is restricted inside a closed hypersurface.
If such a hypersurface minimizes area for small perturbations then it is stable for the corresponding problem.
For example, solutions to the isoperimetric problem, that is, the hypersurface with or without boundary that has least area for a fixed enclosed volume, are in particular stable CMC hypersurfaces.

In \cite{barbosa-docarmo1984, barbosa-docarmo-eschenburg1988} Barbosa-do Carmo and Barbosa-do Carmo-Eschenburg classify stable closed CMC hypersurfaces of Euclidean spaces, spheres and hyperbolic spaces.
A similar result was obtained by Souam \cite{souam1997} for stable free boundary CMC hypersurfaces in a hemisphere and more recently for closed CMC surfaces in $S^2\times \real$ and $\mathbb{H}^2\times\real$.
Other classification results for stable free boundary CMC surfaces were obtained by Ros-Vergasta \cite{ros-vergasta1995} and later improved by Nunes \cite{nunes2017}.
It is also natural to study CMC hypersurfaces of higher index.
That is, those that have some small pertubations that decrease area with fixed enclosed volume.

In \cite{torralbo-urbano2012} Torralbo-Urbano make use of isometric embeddings of homogenous $3$-manifolds into Euclidean space to study stable closed CMC surfaces and the isoperimetric problem in Berger spheres.
The authors also use coordinates of hamornic vector fields to construct test functions for the second variation of area.

In the case of minimal surfaces there has been multiple results establishing a connection between the topology of the surface and its index.
For example, do Carmo-Peng \cite{docarmo-peng1982}, Fischer-Colbrie-Schoen \cite{fischer-colbrie-schoen1980} and Pogorelov \cite{pogorelov1981} have independently proved that stable two-sided minimal surfaces in $\real^3$ are planes.

In \cite{ros2006} Ros proves that the index of a minimal surface in $\real^3$ is bounded below by a linear function of its genus, which was later improved by Chodosh-Maximo \cite{chodosh-maximo2016,chodosh-maximo2018:arxiv}.
A corresponding result was shown by Savo \cite{savo2010} for minimal hypersurfaces in $n$-spheres.
In both situations the authors use harmonic vector fields to construct test functions but their construction is fundamentally distinct.
Ros uses the coordinates of the harmonic vector field with respect to the usual basis of $\real^3$ which works seamlessly since the ambient space is flat.
On the other hand, Savo uses the coordinate with respect to a vector field parametrized by two unit vectors and later takes the average of the second variation for all such test functions.
This allows to overcome the extra term in the second variation given by the curvature of the ambient space.

Using a similar method to Savo's, Sargent \cite{sargent2017} proved a corresponding index estimate for free boundary minimal hypersurfaces in convex bodies of Euclidean space.
In an outstanding work, Ambrozio-Carlotto-Sharp \cite{ambrozio-carlotto-sharp2018.1} used both Ros' and Savo's approach to relate the index of closed minimal hypersurface and its first betti number in an arbitrary ambient manifold that can be suitably embedded into some Euclidean space.
The authors later did the same for free boundary minimal hypersurfaces in $2$-convex domains \cite{ambrozio-carlotto-sharp2018.2}.

Our main results in this paper follow the natural generalization of Torralbo-Urbano for higher index CMC surfaces, similar to Ambrozio-Carlotto-Sharp's approach.
The main difference is that we are only allowed to use admissible functions.
Fortunately, the coordinates of harmonic vector fields are admissible in the closed case.
In the free boundary case we need the extra condition that these vector fields are tangential along the boundary.
In either case we have to use both the harmonic vector field and its Hodge dual, so it only makes sense in the case of CMC surfaces.
It is not clear whether or not Savo's test functions preserve enclosed volume, so the generalization to higher dimensions seem to not be straightforward.

A surprising difference between the CMC case and the minimal case, at least in the case of surfaces, is that the extra term involving the non-zero mean curvature allows us to have a wider variety of applications.
More specifically, some of the ambient spaces that satisfy the conditions of our theorems for non-zero mean curvature do not work in the case of minimal surfaces, for example, flat ambient spaces.

Our main theorem for closed CMC surfaces is:
\begin{maintheorem}{\ref{main theorem closed}}
Let $(M,g)$ be a $3$-dimensional Riemannian manifold without boundary isometrically embedded in $\real^d$ and $\Sigma$ a closed two-sided immersed CMC surface in $M$. 
Assume there exists a real number $\eta$ and a $q$-dimensional vector space $\mathbb{V}^q$ of harmonic vector fields on $\Sigma$ such that any non-zero $\xi\in \mathbb{V}^q$ satisfies
\[\int_{\Sigma}\sum_{i=1}^2\left(|\sffM_M(e_i,\xi)|^2+|\sffM_M(e_i,\star\xi)|^2\right)-(R_{M}+H_\Sigma^2)|\xi|^2\ dV_\Sigma<2\eta \int_{\Sigma}|\xi|^2\ dV_\Sigma.\]
Then
\[ \#  \{ \text{eigenvalues\ of\ $\tilde L_\Sigma$\ that\ are\ strictly\ smaller\ than}\ \eta\} \geq \frac{q}{2d}.\]
\end{maintheorem}

And the corresponding index estimates:
\begin{maincorollary}{\ref{index estimates}}
Let $(M,g)$ be a $3$-dimensional Riemannian manifold without boundary isometrically embedded in $\real^d$ and $\Sigma$ a closed two-sided immersed CMC surface of genus $g$ in $M$.
Suppose that every non-zero $\xi\in \mathcal{H}^1(\Sigma)$ satisfies
\[\int_{\Sigma}\sum_{i=1}^2\left(|\sffM_M(e_i,\xi)|^2+|\sffM_M(e_i,\star\xi)|^2\right)-R_{M}|\xi|^2 \ dV_\Sigma<\int_\Sigma H_\Sigma^2|\xi|^2\ dV_\Sigma.\]
Then
\begin{equation*}
\Index(\Sigma)\geq\frac{g}{d}.
\end{equation*}
\end{maincorollary}

As mentioned above, the free boundary case is slightly different since, a priori, only harmonic vector fields that are tangential along the boundary provide admissible test functions.
However, it is still possible to find tangential vector fields whose dual vector, despite not being tangential any more, has zero average along the boundary.
This can be done as long as the dimension of the space of tangential harmonic vector fields is sufficiently large.
As a consequence the index estimates are fairly weaker.

The respective results for free boundary CMC surfaces are:

\begin{maintheorem}{\ref{main theorem boundary}}
Let $(M,\partial M,g)$ be a $3$-dimensional Riemannian manifold with boundary isometrically embedded in $\real^d$ and $\Sigma$ a compact two-sided immersed free boundary CMC surface in $M$. 
Assume there exists a real number $\eta$ and a $q$-dimensional vector space $\mathbb{W}^q$ of harmonic vector fields on $\Sigma$ that are tangential on the boundary $\partial\Sigma$, such that any non-zero $\xi\in \mathbb{W}^q$ satisfies
\begin{equation*}
\begin{aligned}
\int_{\Sigma} & \sum_{i=1}^2|\sffM_M(e_i,\xi)|^2+|\sffM_M(e_i,\star\xi)|^2\ dV_\Sigma\\
& -\int_\Sigma (R_{M} +H_\Sigma^2)|\xi|^2\ dV_\Sigma -2\int_{\partial \Sigma}H_{\partial M}|\xi|^2 \ dV_{\partial \Sigma} <2\eta \int_{\Sigma}|\xi|^2\ dV_\Sigma.
\end{aligned}
\end{equation*}
Then
\[ \#  \{ \text{eigenvalues\ of\ $\tilde L_\Sigma$\ that\ are\ smaller\ than}\ \eta\} \geq \frac{q-d}{2d}.\]
\end{maintheorem}

\begin{maincorollary}{\ref{index estimate free boundary}}
Let $(M,\partial M,g)$ be a $3$-dimensional Riemannian manifold with boundary isometrically embedded in $\real^d$ and $\Sigma$ a compact two-sided immersed free boundary CMC surface in $M$ with genus $g$ and $r$ boundary components. 
Suppose that every non-zero $\xi\in \harmonicvfb(\Sigma,\partial \Sigma)$ satisfies
\begin{equation*}
\begin{aligned}
\int_{\Sigma} & \sum_{i=1}^2|\sffM_M(e_i,\xi)|^2+|\sffM_M(e_i,\star\xi)|^2\ dV_\Sigma\\
 & -\int_\Sigma R_{M} |\xi|^2\ dV_\Sigma -2\int_{\partial \Sigma}H_{\partial M}|\xi|^2 \ dV_{\partial \Sigma} < \int_\Sigma H_\Sigma^2|\xi|\ dV_\Sigma.
\end{aligned}
\end{equation*}
Then
\[\Index(\Sigma)\geq\frac{2g+r-1-d}{2d}.\]
\end{maincorollary}

In the Theorems above the operator $\tilde L_\Sigma$ corresponds to the twisted Dirichlet eigenvalue problem for the Jacobi operator on admissible functions.
The number of negative eigenvalues is the index of $\Sigma$ in the CMC sense.

Let us mention that Cavalcande-de Oliveira in \cite{cavalcante-oliveira2017:arxiv,cavalcante-oliveira2018:arxiv} have obtained similar estimates independent of the inequality condition on the embedding of the ambient manifold for the particular case of closed CMC surfaces in $\real^3$, $S^3$ and free boundary CMC surfaces in mean convex domains of these two cases.

A conjecture attributed to Schoen and Marques-Neves (see \cite{marques2014, neves2014}) which says that minimal surfaces in $3$-manifolds of positive Ricci curvature have index bounded below by a linear function of its genus.
In \cite{ambrozio-carlotto-sharp2018.1} Ambrozio-Carlotto-Sharp have confirmed this conjecture in a large class of ambient spaces.
One may ask a similar question for closed CMC surfaces.
In this case, it follows from our results that any closed CMC surface of sufficiently large mean curvature, depending only on the ambient space, in any closed $3$-manifold has index bounded below by a linear function of the genus.
In some examples we compute explicitly the lower bound necessary for the mean curvature.

As another application, in some examples we are able to prove that closed stable CMC surfaces of sufficiently large mean curvature have to be spheres.
In particular, in $T^2\times\real$ and in $T^3$ this partially supports a conjecture of Hauswirth-Perez-Romon-Ros regarding the isoperimetric profile of such ambient spaces.
To be more specific, the authors conjecture that the isoperimetric profile is given by spheres, cylinders and pairs of planes.
Since isoperimetric surfaces are stable CMC surfaces, our result supports the section of the profile that is given by spheres.

This paper is divided as follows. 
In section $2$ we cover some of the necessary background and notations used throughout the paper.
In section $3$ we prove that Ros' test functions are admissible and prove our main theorems.
In section $4$ we discuss some applications and examples that satisfy the hypothesis of our main results.

\hfill

\textit{
Acknowledgements: The authors would like to thank Professors Jingyi Chen and Ailana Fraser for helpful conversations and comments on this work.
We are also thankful to Professor Marcos Cavalcante for corrections on an earlier version of this work and Francisco Torralbo for pointing out further applications.
}

\section{Preliminaries}
In this section, we will overview some background and establish notation that will be used throughout the article.

Suppose $(M,\partial M,g)$ is a complete Riemannian manifold of dimension $n+1$, with possibly empty boundary and let $(\Sigma^k, \partial \Sigma)$ be a compact immersed $k$-dimensional submanifold in $M$ with boundary $\partial \Sigma$ in $\partial M$.
In case $\partial M$ is empty, we consider $\partial \Sigma$ to also be empty, in which case we say $\Sigma^k$ is a closed (compact without boundary) immersed submanifold in $M$. 
When $k=n$, we say that $\Sigma^n$ is two-sided if there is a global unit normal vector field $N$ along $\Sigma$ in $M$. 
Denote the connection on $M$ and $\Sigma$ by $\connM$ and $\connS$, respectively.
The second fundamental form of $\Sigma$ is defined by $A_\Sigma(X,Y) = \connM_X Y-\connS_X Y$, its shape operator with respect to $N$ is denoted $S_\Sigma(X)$, the mean curvature vector is given by $\vec{H}_\Sigma=\tr_\Sigma(A_\Sigma)$ and its mean curvature is $H_\Sigma=g(\vec{H}_\Sigma,N)$. 

By Nash's Embedding theorem we may assume $M$ to be isometrically embedded into some Euclidean space $\real^d$ for some integer $d$.
If we denote the Euclidean connection by $D$, then the second fundamental form of $M$ in $\real^d$ is given by $\sffM_M(X,Y)=D_X Y -\connM_X Y$, for any two tangent vectors $X,Y$ on $M$.
It follows that for any two tangent vector $X,Y$ on $\Sigma$, the following orthogonal decomposition holds
\begin{equation}\label{decomposition2}
D_X Y=\connS_X Y+ A_\Sigma(X,Y)+\sffM_M(X,Y).
\end{equation}

Denote by $\Delta^{[p]}$ the Hodge Laplacian of $\Sigma$ acting on $p$-forms and let $\nabla^*\nabla$ denote the rough Laplacian on vector fields. 
By choosing a local geodesic frame $\{e_1,\cdots,e_n\}$ of $\Sigma$, they are defined by
\[\Delta^{[p]} w=(d\delta+\delta d)w, \ \ \nabla^*\nabla X=-\sum_{i=1}^n\nabla_{e_i}\nabla_{e_i}X\]
for $p$-forms $w$ and vector fields $X$ on $\Sigma$, where $d$ is the exterior differential and $\delta$ is the codifferential. 
If $\Sigma$ is closed, then $w$ is harmonic if and only if $dw=0$ and $\delta w=0$.

Given a $1$-form $w$, its dual vector field $w^{\natural}$ is defined by $g(w^{\natural},X)=w(X)$ for any vector $X$.
Reversely, given a vector field $X$ we denote its dual $1$-form $X^\flat(Y)=g(X,Y)$.

Let us now restrict ourselves to when $\Sigma$ is a surface.
We may define $\star X= (\star X^\flat)^\natural$, where $\star$ is the Hodge operator with respect to the metric on $\Sigma$.
The Laplacian acting on a vector field $X$ may be defined as $\Delta_{[1]}X=\Delta^{[1]}X^\flat$.
If $X$ is a vector field on $\Sigma$, we have $\star dX^\flat=\divergent_\Sigma \star X$ and $\delta X^\flat=-\divergent_\Sigma X$.
Hence, when $\Sigma$ is closed, $X$ is harmonic if and only if $\divergent_\Sigma X=\divergent_\Sigma \star X =0$.

When $\Sigma$ is a closed surface we denote the space of harmonic vector fields on $\Sigma$ as
\[\harmonicvf(\Sigma)=\{\xi\in T\Sigma:\ \divergent(\xi)=\divergent(\star \xi)=0 \ \text{on}\ \Sigma\ \}.\]
If $\Sigma$ has genus $g$ then $\dimension \harmonicvf(\Sigma)=2g$.

When $\Sigma$ has nonempty boundary, a vector field being harmonic is not equivalent to $\divergent(\xi)=\divergent(\star \xi)=0$. 
Instead we shall consider the following space:
\[\harmonicvfb(\Sigma,\partial \Sigma)=\{\xi \in T\Sigma:\ \divergent(\xi)=\divergent(\star \xi)=0 \ \text{on}\ \Sigma\ \text{and}\ \xi \text{ is tangent along}\ \partial \Sigma\}.\]
Any vector field in this space is harmonic and tangential along $\partial \Sigma$. 
By Hodge theorem and Poincar\'e-Lefschetz duality, we know that $\harmonicvfb(\Sigma,\partial \Sigma)$ is isomorphic to $H_1(\Sigma,\partial\Sigma; \mathbb{R})$. 
Hence $\dimension \harmonicvfb(\Sigma,\partial \Sigma)=2g+r-1$, where $g$ is the genus and $r$ is the number of boundary components of $\Sigma$.
See for example \cite[Lemma A.0.1]{sargent2017} for a proof.

It is well know that Weitzenbock's formula relates the Hodge Laplacian and rough Laplacian of a vector field. 
That is, for a vector field $\xi$ on $\Sigma$ we have
\begin{equation}\label{bochner}
\Delta_{[1]}\xi=\nabla^*\nabla \xi+\operatorname{Ric}_\Sigma(\xi),
\end{equation}
where $\operatorname{Ric}_\Sigma(\xi)$ is defined by $\operatorname{Ric}_\Sigma(\xi,X)=g(\operatorname{Ric}_\Sigma(\xi),X)$ for every tangent vector $X$.

Now, let us define the Morse index of a CMC hypersurface.
Let $\Sigma$ be a two-sided hypersurface with boundary $\partial \Sigma$ in $(M,\partial M)$.
The first variation formula of area with respect to a normal variation induced by a function $u$ on $\Sigma$ is given by 
\begin{equation}\label{firstvariation}
\frac{dV_t}{dt}\Big|_{t=0}=-n\int_\Sigma uH_\Sigma\ dV_\Sigma+\int_{\partial \Sigma} g(\eta, X)\ dV_{\partial \Sigma},
\end{equation}
where $\eta$ is the outward pointing unit normal vector field along $\partial\Sigma$ and $X$ is the orthogonal projection of $uN$ onto $\partial M$.
A free boundary constant mean curvature hypersurface, henceforth denoted free boundary CMC hypersurface, is a critical point of the area functional with respect to variations that preserve enclosed volume. 
That is, $u$ must satisfy $\int_{\Sigma}u\ dV_\Sigma=0$ and we call it an admissible variation. 
Note that $\partial \Sigma$ must intersect $\partial M$ orthogonally, which is called the free boundary property, i.e., $\eta=-\nu$ where $\nu$ is the inward pointing normal vector field along $\partial M.$  

The quadratic form associated to the second variation of area of a free boundary CMC hypersurface with respect to $u$ is 
\begin{equation}\label{variation1}
Q_\Sigma(u,u)=\int_\Sigma |\nabla u|^2-(\operatorname{Ric}_M(N,N)+|A_\Sigma|^2)u^2\ dV_\Sigma-\int_{\partial \Sigma}h_{\partial M}(N,N)\ dV_{\partial \Sigma},
\end{equation}
where $\operatorname{Ric}_M(N,N)$ is the Ricci curvature of $M$ in the direction $N$, $|A_\Sigma|^2=\tr_\Sigma(S_\Sigma^TS_\Sigma)$ is the square norm of second fundamental form and $h_{\partial M}(X,Y)=g(\connM_X Y,\nu)$ denotes the scalar second fundamental form of $\partial M$ in $M$ with respect to $\nu$, i.e., 
It follows from the free boundary property that $h_{\partial M}(N,N)$ in $(\ref{variation1})$ is well defined. 

By the divergence theorem we can also write $(\ref{variation1})$ as 
\begin{equation*}
Q_\Sigma(u,u)=-\int_{\Sigma} u\Delta u+(\operatorname{Ric}_M(N,N)+|A_\Sigma|^2)u^2\ dV_\Sigma+\int_{\partial \Sigma}u\frac{\partial u}{\partial \eta}-h_{\partial M}(N,N)u^2\ dV_{\partial \Sigma}.
\end{equation*}

We denote by $L_\Sigma=\Delta+\operatorname{Ric}_M(N,N)+|A_\Sigma|^2$ the Jacobi operator of $\Sigma$, where $\Delta u = \divergent_\Sigma(\connS u)$. 

The Morse index of a free boundary CMC hypersurface is defined as the index of the quadratic form $Q_\Sigma$ restricted to the subspace of smooth functions $u$ on $\Sigma$ that are admissible, that is, $\int_\Sigma u dV_\Sigma=0$.
When the index is zero, we say that the hypersurface is stable in the CMC sense.

This notion of index coincides with the number of negative eigenvalues for the operator $\tilde{L}_\Sigma u = L_\Sigma u -\frac{1}{|\Sigma|}\int_{\Sigma}L_\Sigma u dV_\Sigma$ acting on admissible functions on $\Sigma$ (see \cite[Proposition 2.2]{barbosa-berard2000}).
That is, the following eigenvalue problem:
\begin{equation}\label{equations}
\left\{
   \begin{aligned}
    & \tilde{L}_\Sigma u + \lambda u = 0, && \text{in}\ \Sigma\\
    & \frac{\partial u}{\partial \eta}-h_{\partial M}(N,N)u=0, && \text{on}\ \partial \Sigma \\
    & \int_\Sigma u dV_\Sigma=0.
   \end{aligned}\right.
\end{equation}
The spectrum of $\tilde{L}_\Sigma$ consists of eigenvalues $\tilde\lambda_1\leq\tilde\lambda_2\leq\ldots$ corresponding to admissible eigenfunctions $\tilde\phi_1,\tilde\phi_2,\ldots$ in $L^2(\Sigma)$, see \cite{barbosa-berard2000} for details.
Furthermore, the eigenvalues $\tilde\lambda_i$ satisfy the min-max characterization with respect to the quadratic form $Q_\Sigma$ when restricted to admissible functions.
That is, $\lambda_k=\inf \frac{Q_\Sigma(u,u)}{\int_\Sigma u^2 dV_\Sigma}$, where the infimum is taken over all smooth admissible functions that are orthogonal to the first $k-1$ eigenfunctions $\tilde\phi_1,\ldots,\tilde\phi_{k-1}$.

In the special case where the boundaries $\partial M$ and $\partial \Sigma$ are both empty, the corresponding variation formulas are obtained similarly by simply making the terms that involve the boundary as $0$.
 
Similarly we can define the Morse index of a closed CMC hypersurface as the number of negative eigenvalues of the Jacobi operator with respect to admissible variations. 
Note that the Jacobi operators are the same in closed and boundary case, the only difference is the boundary condition in the eigenvalue problem.

\begin{remark}
Throughout this paper we always assume that $\Sigma$ is a two-sided surface.
Nevertheless, every theorem has an analogous statement when $\Sigma$ is one-sided.
The proofs are similar but instead of working with $\Sigma$ we use its two-sided double cover immersion $\tilde{\Sigma}$.
A variation on $\tilde\Sigma$ is given by a function that is anti-invariant with respect to the involution that changes sheets.
In this case we consider harmonic $1$-forms that are again anti-invariant under the action of the involution.
The space of such forms has dimension equal to $\tilde g$, which is the genus of $\tilde\Sigma$ in the closed case.
Similarly, in the free boundary case, the corresponding space of tangential anti-invariant harmonic $1$-forms has dimension $\tilde g + r-1$  where $r$ is the number of ends of $\Sigma$.
\end{remark}

\section{Main results}
Similarly to \cite{ros2006} and \cite{ambrozio-carlotto-sharp2018.1,ambrozio-carlotto-sharp2018.2} we intend to use the coordinates of harmonic vector fields as test functions to obtain our index estimates.
However, we must first verify that such coordinates are admissible in the CMC sense.

The following is a general result in arbitrary codimension that follows from the divergence theorem.
\begin{lemma}\label{findtestfunctions}
Suppose $(\Sigma^k,\partial\Sigma, g)$ is a $k$-dimensional submanifold $(k\geq 2)$ isometrically embedded in some Euclidean space $\real^d$. 
\begin{itemize}
\item If $\partial \Sigma=\emptyset$ and $\xi$ is a harmonic vector field in $\harmonicvf(\Sigma)$, then 
\[\int_{\Sigma}\langle V,\xi\rangle \ dV_\Sigma=0,\]
where $V$ is a constant vector in $\real^d$. 
In particular, if $k=2$, we also have
\[\int_{\Sigma}\langle V,\star\xi\rangle\ dV_\Sigma=0.\]
\item If $\partial \Sigma\neq \emptyset$ and $\xi$ is a tangential harmonic vector field in $\harmonicvfb(\Sigma,\partial\Sigma)$, then
\[\int_{\Sigma}\langle V,\xi\rangle \ dV_\Sigma=0.\]
\end{itemize}
\end{lemma}

\begin{proof}
Denote by $f(x)=\langle V,x\rangle$, then its gradient is given by $\connS f=V^T$, where $V^T$ denotes the projection of $V$ onto $\Sigma$.
If $\Sigma$ is a closed submanifold and $\xi\in \harmonicvf(\Sigma)$ we get \[0=\int_\Sigma \operatorname{div}_\Sigma(f\xi)=\int_\Sigma \langle \nabla^\Sigma f,\xi\rangle+f\operatorname{div}(\xi)=\int_\Sigma \langle \nabla^\Sigma f,\xi\rangle=\int_\Sigma \langle V,\xi\rangle.\]
When $k=2$, the argument follows by substituting $\star \xi$ and observing that it is also a harmonic vector field. 

In the case where $\Sigma$ has nonempty boundary and $\xi\in \harmonicvfb(\Sigma)$ we get
\[\int_\Sigma \langle V,\xi\rangle=\int_\Sigma \operatorname{div}_\Sigma(f\xi)=\int_{\partial \Sigma}f\langle\xi,\eta\rangle=0\]
The last equality is due to the tangential property of $\xi$. 
This completes the proof of the Lemma.
\end{proof}

\begin{remark}\label{remark test functions}
When $\Sigma$ has nonempty boundary and dimension $2$ this argument does not work for $\star \xi$ since it may not be tangential on the boundary when $\xi$ is tangential. 
In fact, if $\xi$ is tangential along the boundary then $\star \xi$ is necessarily orthogonal to the boundary.
However, we may still be able to find tangential harmonic vector fields $\xi$ such that $f\langle\xi,\eta\rangle$ has zero average along the boundary so its coordinates may still be admissible test functions.
\end{remark}

The following proposition involves calculations of coordinates of a harmonic vector field as test functions for the second variation formula both in closed case and boundary case. 
In fact, the calculations in closed case directly follow from the result in boundary case.
\begin{proposition}\label{calculation}
Let $(M,\partial M,g)$ be a $3$-dimensional Riemannian manifold isometrically embedded in some Euclidean space $\real^d$. 
Let $(\Sigma,\partial \Sigma)$ be a two-sided immersed surface in $M$ with free boundary $\partial \Sigma$ in $\partial M$.

\begin{itemize}
\item  Given a vector field $\xi\in \harmonicvfb(\Sigma,\partial \Sigma)$, denote
\[u_j=\langle \xi,E_j\rangle\]
where $\{E_j\}_{j=1}^{d}$ is the canonical basis of $\real^d$. Then
\begin{equation}\label{plugin}
\begin{aligned}
\sum_{j=1}^{d}Q_\Sigma(u_j,u_j) = &\int_{\Sigma}\sum_{i=1}^2|\sffM_M(e_i,\xi)|^2+|A_\Sigma(e_i,\xi)|^2\ dV_{\Sigma}\\
   & \quad - \int_\Sigma \left(\frac{|A_\Sigma|^2}{2}+\frac{\scalar_{M}}{2}+\frac{H_\Sigma^2}{2}\right)|\xi|^2\ dV_\Sigma-\int_{\partial\Sigma}H_{\partial M}|\xi|^2\ dV_{\partial \Sigma} 
\end{aligned}
\end{equation}
    where $\{e_i\}_{i=1}^2$ is an orthonormal frame on $\Sigma$, $\scalar_M$ is the scalar curvature of $M$ and $H_{\partial M}$ denotes the mean curvature of $\partial M$ with respect to the inner normal vector $\nu$.
\item  In particular, if $\partial M$ and $\partial \Sigma$ are empty, for any vector field $\xi\in \mathcal{H}^1(\Sigma)$, we have
\begin{equation}\label{plugin2}
\begin{aligned}
\sum_{j=1}^{d}Q_\Sigma(u_j,u_j) = &\int_{\Sigma}\sum_{i=1}^2|\sffM_M(e_i,\xi)|^2+|A_\Sigma(e_i,\xi)|^2\ dV_{\Sigma}\\
   & \quad - \int_\Sigma \left(\frac{|A_\Sigma|^2}{2}+\frac{\scalar_{M}}{2}+\frac{H_\Sigma^2}{2}\right)|\xi|^2\ dV_\Sigma
\end{aligned}
\end{equation}
\end{itemize}
\end{proposition}

\begin{proof}
Using a local orthonormal basis $\{e_i\}_{i=1}^2$ on $\Sigma$ we have
\[\nabla^\Sigma u_j=\sum_{i=1}^2e_i\langle \xi,E_j\rangle e_i=\sum_{i=1}^2\langle D_{e_i}\xi,E_j\rangle e_i.\]
 Then $(\ref{decomposition2})$ gives
\[D_{e_i}\xi=\nabla^\Sigma_{e_i}\xi+A_\Sigma(e_i,\xi)+\sffM_M(e_i,\xi).\]
Thus,
\begin{eqnarray}\label{111}
\sum_{j=1}^d|\nabla^{\Sigma} u_j|^2&=& \sum_{j=1}^d\sum_{i=1}^2|\langle D_{e_i}\xi,E_j\rangle|^2\\
&=& \sum_{i=1}^2\sum_{j=1}^d \langle \nabla^\Sigma_{e_i}\xi,E_j\rangle^2+\langle A_\Sigma(e_i,\xi),E_j\rangle^2+\langle \sffM_M(e_i,\xi),E_j\rangle^2\nonumber\\
&=& |\nabla^\Sigma \xi|^2+\sum_{i=1}^2| A_\Sigma(e_i,\xi)|^2+ |\sffM_M(e_i,\xi)|^2.\label{gradient}
\end{eqnarray}
Gauss' equation for $\Sigma$ in $M$ gives us
\[2K_\Sigma=\scalar_M-2\operatorname{Ric}_M(N,N)-|A_\Sigma|^2+H_\Sigma^2.\]
Hence,
\begin{equation}\label{Ric}\int_\Sigma \operatorname{Ric}_M(N,N)u_j^2 \ dV_\Sigma=\int_\Sigma \left(\frac{\scalar_M}{2}-\frac{|A_\Sigma|^2}{2}+\frac{H_\Sigma^2}{2}-K_\Sigma\right)u_j^2\ dV_\Sigma.\end{equation}

It follows from Weitzenbock's formula (\ref{bochner}), $\xi$ being harmonic vector field and $\Sigma$ being surface that
\[\nabla^*\nabla \xi=-K_\Sigma\xi.\]
By computing the exterior derivative along $\partial \Sigma$ we have
\[d\xi^\flat(\eta,\xi)=\langle \nabla^\Sigma_\eta\xi,\xi\rangle-\langle \nabla^\Sigma_\xi \xi,\eta\rangle.\]
Since $d\xi^\flat=0$,
\[\langle\nabla_\eta^\Sigma\xi,\xi\rangle=\langle\nabla_\xi^\Sigma\xi,\eta\rangle.\]
Then It follows from the divergence theorem that
\begin{eqnarray}\label{meancurvature}
\int_\Sigma \Delta|\xi|^2 \ dV_\Sigma&=& \int_{\partial \Sigma}\nabla^\Sigma_\eta |\xi|^2\ dV_{\partial \Sigma}\nonumber\\
&=& 2\int_{\partial \Sigma}\langle\nabla_\eta^\Sigma\xi,\xi\rangle\ dV_{\partial \Sigma}\nonumber\\
&=& 2\int_{\partial \Sigma}\langle\nabla_\xi^\Sigma\xi,\eta\rangle\ dV_{\partial \Sigma}\nonumber\\
&=& -2\int_{\partial \Sigma}\langle \connM_\xi\xi,\nu\rangle\ dV_{\partial \Sigma}\nonumber\\
&=&-2\int_{\partial \Sigma}h_{\partial M}(\xi,\xi)\ dV_{\partial \Sigma},
\end{eqnarray}
which together with 
\[\Delta |\xi|^2=-2\langle\nabla^*\nabla\xi,\xi\rangle+2|\nabla^\Sigma \xi|^2\]
implies that
\begin{eqnarray}\label{boundaryterm}\int_\Sigma|\nabla^\Sigma \xi|^2\ dV_\Sigma&=&\int_{\Sigma}\frac{\Delta|\xi|^2}{2}\ dV_\Sigma+\int_{\Sigma}\langle \nabla^*\nabla\xi,\xi\rangle\ dV_\Sigma\nonumber\\
&=&-\int_{\partial \Sigma}h_{\partial M}(\xi,\xi)\ dV_{\partial \Sigma}-\int_\Sigma K_\Sigma|\xi|^2\ dV_\Sigma.\end{eqnarray}
Note that
\begin{equation}\label{boundarycurvature}H_{\partial M}|\xi|^2=h_{\partial M}(\xi,\xi)+h_{\partial M}(N,N)|\xi|^2.\end{equation}
Then $(\ref{plugin})$ in the proposition follows from $(\ref{gradient})$, $(\ref{Ric})$,  $(\ref{boundaryterm})$, $(\ref{boundarycurvature})$ and $(\ref{variation1})$.  
The formula $(\ref{plugin2})$ directly follows by ignoring boundary terms.
\end{proof}

\begin{remark}\label{remark1}
Since $|\star \xi|=|\xi|$, we have same result for $\star\xi$ as $(\ref{meancurvature})$ that
\begin{equation}\label{starmeancurvature}
\int_\Sigma \Delta|\star\xi|^2=\int_\Sigma \Delta|\xi|^2 \ dV_\Sigma=\int_{\partial \Sigma}h_{\partial M}(\xi,\xi)\ dV_{\partial \Sigma}.
\end{equation}
It means that we can have same equation as $(\ref{plugin})$ and $(\ref{plugin2})$ for $\star\xi$ whether or not its coordinate functions are admissible.
\end{remark}

The following is the main theorem for closed CMC surfaces in $3$-dimensional manifolds embedded in Euclidean space.
It estimates the number of eigenvalues of the $\tilde L_\Sigma$ below a certain threshold $\eta$ under certain conditions on the embedding and the geometry of $M$.
This result is often referred to as a ``concentration of spectrum inequality'' (see \cite{ambrozio-carlotto-sharp2018.1}).
Recall that $\tilde L_\Sigma u=L_\Sigma u-\int_\Sigma L_\Sigma u dV_\Sigma$.

\begin{theorem}\label{main theorem closed}
Let $(M,g)$ be a $3$-dimensional Riemannian manifold without boundary isometrically embedded in $\real^d$ and $\Sigma$ a closed two-sided immersed CMC surface in $M$. 
Assume there exists a real number $\eta$ and a $q$-dimensional vector space $\mathbb{V}^q$ of harmonic vector fields on $\Sigma$ such that any non-zero $\xi\in \mathbb{V}^q$ satisfies
\[\int_{\Sigma}\sum_{i=1}^2\left(|\sffM_M(e_i,\xi)|^2+|\sffM_M(e_i,\star\xi)|^2\right)-(R_{M}+H_\Sigma^2)|\xi|^2\ dV_\Sigma<2\eta \int_{\Sigma}|\xi|^2\ dV_\Sigma.\]
Then
\[ \#  \{ \text{eigenvalues\ of\ $\tilde L_\Sigma$\ that\ are\ strictly\ smaller\ than}\ \eta\} \geq \frac{q}{2d}.\]
\end{theorem}
\begin{proof}
Firstly, from Lemma $\ref{findtestfunctions}$ we have that $u_j$ and $u_j^\star=\langle\star\xi, E_j\rangle$ for $j=1,\cdots,d$ are admissible test functions. 
Applying the same calculation for $(\ref{plugin2})$ to $u_j^{\star}$ we have that
\begin{equation}\label{dual}
\begin{aligned}
\sum_{j=1}^{d}Q_\Sigma(u_j^\star,u_j^\star) = &\int_{\Sigma}\sum_{i=1}^2|\sffM_M(e_i,\star\xi)|^2+|A_\Sigma(e_i,\star\xi)|^2\ dV_{\Sigma}\\
&\quad  - \int_\Sigma \left(\frac{|A_\Sigma|^2}{2}+\frac{\scalar_{M}}{2}+\frac{H_\Sigma^2}{2}\right)|\star\xi|^2\ dV_\Sigma.
\end{aligned}
\end{equation}

Now, whenever $\xi\neq 0$ we may pick $e_1=\frac{\xi}{|\xi|}, e_2=\frac{\star\xi}{|\star\xi|}$ as an orthonormal basis.
Recalling that $|\xi|=|\star\xi|$, we have at every point
\begin{equation}\label{keystep1}
\sum_{i=1}^2 |A_\Sigma(e_i,\xi)|^2+|A_\Sigma(e_i,\star\xi)|^2 = \sum_{i,j=1,2}|A_\Sigma(e_i,e_j)|^2|\xi|^2 = |A_\Sigma|^2|\xi|^2.
\end{equation}
From which we obtain, by adding $(\ref{plugin2})$ and $(\ref{dual})$,
\begin{equation*}
\begin{aligned}
\sum_{j=1}^{d}Q_\Sigma(u_j,u_j)+Q_\Sigma(u^\star_j,u^\star_j) = &\int_{\Sigma}\sum_{i=1}^2|\sffM_M(e_i,\xi)|^2+|\sffM_M(e_i,\star\xi)|^2\ dV_\Sigma\\
 &\quad -\int_{\Sigma}(R_{M} +H_\Sigma^2)|\xi|^2\ dV_\Sigma\nonumber.
\end{aligned}
\end{equation*}

For convenience, denote by $k=\#\{ \text{eigenvalues\ of\ $\tilde L_\Sigma$\ that\ are\ strictly\ smaller\ than}\ \eta\}$ and by $\tilde\phi_1,\cdots,\tilde\phi_k$ the eigenfunctions of the operator $\tilde L_\Sigma$ with respect to eigenvalues $\tilde \lambda_1\leq\cdots\leq\tilde\lambda_k<\eta$.
Consider following the linear map defined by
\begin{eqnarray*}
F: &\mathbb{V}^q\longrightarrow &\mathbb{R}^{2dk}\\
& \xi \longmapsto &[\int_\Sigma u_j\tilde \phi_\alpha\ dV_\Sigma, \int_\Sigma u^\star_j\tilde\phi_\alpha\ dV_\Sigma ],
\end{eqnarray*}
where $\alpha=1,\cdots,k$ and $j=1,\cdots,d$.
By Rank-Nullity theorem, if $2dk<q$, then there exists a nonzero harmonic vector field $\xi$ such that $u_j,u_j^\star$ are orthogonal to the first $k$ eigenfunctions of $\tilde L_\Sigma$.
From the Min-max principle for the operator $\tilde L_\Sigma$ it follows that
\begin{equation*}\sum_{j=1}^{d}Q_\Sigma(u_j,u_j)+Q_\Sigma(u^\star_j,u^\star_j)\geq 2\lambda_{k+1}\int_\Sigma|\xi|^2 \ dV_\Sigma\geq 2\eta\int_\Sigma |\xi|^2\ dV_\Sigma\end{equation*}
which contradicts the assumption in the proposition. Thus $2dk\geq q$, that is, $k\geq \frac{q}{2d}$ as claimed.
\end{proof}

We remark that the same proof works to estimate the number of eigenvalues of the Jacobi operator.
However, such estimate would not imply index estimates in the CMC sense.

As a corollary, letting $\eta=0$ we can get an estimate on numbers of negative eigenvalues of $\tilde L_\Sigma$, that is, an index estimate.
\begin{corollary}\label{index estimates}
Let $(M,g)$ be a $3$-dimensional Riemannian manifold without boundary isometrically embedded in $\real^d$ and $\Sigma$ a closed two-sided immersed CMC surface of genus $g$ in $M$.
Suppose that every non-zero $\xi\in \mathcal{H}^1(\Sigma)$ satisfies
\[\int_{\Sigma}\sum_{i=1}^2\left(|\sffM_M(e_i,\xi)|^2+|\sffM_M(e_i,\star\xi)|^2\right)-R_{M}|\xi|^2 \ dV_\Sigma<\int_\Sigma H_\Sigma^2|\xi|^2\ dV_\Sigma.\]
Then,
\begin{equation*}
\Index(\Sigma)\geq\frac{g}{d}
\end{equation*}
\end{corollary}

The next result is the corresponding theorem for free boundary CMC surfaces in $3$-dimensional manifolds with boundary.
From Remark $\ref{remark test functions}$ we observe $u_j^\star$ may not be admissible, however, it is possible to find suitable vector fields when the dimension of $\harmonicvfb(\Sigma,\partial \Sigma)$ is sufficiently large.
\begin{theorem}\label{main theorem boundary}
Let $(M,\partial M,g)$ be a $3$-dimensional Riemannian manifold with boundary isometrically embedded in $\real^d$ and $\Sigma$ a compact two-sided immersed free boundary CMC surface in $M$. 
Assume there exists a real number $\eta$ and a $q$-dimensional vector space $\mathbb{W}^q$ of harmonic vector fields on $\Sigma$ that are tangential on the boundary $\partial\Sigma$, such that any non-zero $\xi\in \mathbb{W}^q$ satisfies
\begin{equation*}
\begin{aligned}
\int_{\Sigma} & \sum_{i=1}^2|\sffM_M(e_i,\xi)|^2+|\sffM_M(e_i,\star\xi)|^2\ dV_\Sigma\\
& -\int_\Sigma (R_{M} +H_\Sigma^2)|\xi|^2\ dV_\Sigma -2\int_{\partial \Sigma}H_{\partial M}|\xi|^2 \ dV_{\partial \Sigma} <2\eta \int_{\Sigma}|\xi|^2\ dV_\Sigma.
\end{aligned}
\end{equation*}
Then
\[ \#  \{ \text{eigenvalues\ of\ $\tilde L_\Sigma$\ that\ are\ smaller\ than}\ \eta\} \geq \frac{q-d}{2d}.\]
\end{theorem}
\begin{proof}
From Remark $\ref{remark1}$ we know that
\begin{equation}\label{plugin3}
\begin{aligned}
\sum_{j=1}^{d}Q_\Sigma(u_j^\star,u_j^\star) = & \int_{\Sigma}\sum_{i=1}^2|\sffM_M(e_i,\star\xi)|^2+|A_\Sigma(e_i,\star\xi)|^2\ dV_{\Sigma}\\
&\quad  - \int_\Sigma \left(\frac{|A_\Sigma|^2}{2}+\frac{\scalar_{M}}{2}+\frac{H_\Sigma^2}{2}\right)|\xi|^2\ dV_\Sigma-\int_{\partial\Sigma}H_{\partial M}|\xi|^2\ dV_{\partial \Sigma}.
\end{aligned}
\end{equation}
Summing up $(\ref{plugin3})$ and $(\ref{plugin})$, together with $(\ref{keystep1})$ gives 
\begin{equation*}
\begin{aligned}
\sum_{j=1}^{d}Q_\Sigma(u_j,u_j)+Q_\Sigma(u^\star_j,u^\star_j) = &\int_{\Sigma}\sum_{i=1}^2|\sffM_M(e_i,\xi)|^2+|\sffM_M(e_i,\star\xi)|^2\ dV_\Sigma\\
 &\quad - \int_{\Sigma}(R_{M} +H_\Sigma^2)|\xi|^2\ dV_\Sigma\nonumber-2\int_{\partial \Sigma}H_{\partial M}|\xi|^2 \ dV_{\partial \Sigma}.
\end{aligned}
\end{equation*}
Let us denote by $k=\#\{ \text{eigenvalues\ of\ $\tilde L_\Sigma$\ that\ are\ smaller\ than}\ \eta\}$ and by $\tilde\phi_1,\cdots,\tilde\phi_k$ the eigenfunctions of $\tilde L_\Sigma$ with respect to eigenvalues $\tilde\lambda_1\leq\cdots\leq\tilde\lambda_k<\eta$.
Consider following the linear map defined by
\begin{eqnarray*}
F: &\mathbb{W}^q\longrightarrow &\mathbb{R}^{2dk+d}\\
& \xi \longmapsto &[\int_\Sigma u_j\phi_\alpha\ dV_\Sigma, \int_\Sigma u^\star_j\phi_\alpha\ dV_\Sigma,\int_{\Sigma} u_j^\star\ dV_\Sigma ],
\end{eqnarray*}
where $\alpha=1,\cdots,k$ and $j=1,\cdots,d$.
By Rank-Nullity theorem, if $2dk+d<q$, then there exists a nonzero harmonic tangential vector field $\xi$ such that $u_j,u_j^\star$ are orthogonal to the first $k$ eigenfunctions of $\tilde L_\Sigma$.
From the Min-max principle for $\tilde L_\Sigma$ it follows that
\[\sum_{j=1}^{d}Q_\Sigma(u_j,u_j)+Q_\Sigma(u^\star_j,u^\star_j)\geq 2\lambda_{k+1}\int_\Sigma|\xi|^2 \ dV_\Sigma\geq 2\eta\int_\Sigma |\xi|^2\ dV_\Sigma\]
which contradicts the assumption in the proposition. Thus $2dk+d\geq q$, that is, $k\geq \frac{q-d}{2d}$ as claimed.
\end{proof}

Correspondingly, we have an index estimate for free boundary CMC surface.
\begin{corollary}\label{index estimate free boundary}
Let $(M,\partial M,g)$ be a $3$-dimensional Riemannian manifold with boundary isometrically embedded in $\real^d$ and $\Sigma$ a compact two-sided immersed free boundary CMC surface in $M$ with genus $g$ and $r$ boundary components. 
Suppose that every non-zero $\xi\in \harmonicvfb(\Sigma,\partial \Sigma)$ satisfies
\begin{equation*}
\begin{aligned}
\int_{\Sigma} & \sum_{i=1}^2|\sffM_M(e_i,\xi)|^2+|\sffM_M(e_i,\star\xi)|^2\ dV_\Sigma\\
 & -\int_\Sigma R_{M} |\xi|^2\ dV_\Sigma -2\int_{\partial \Sigma}H_{\partial M}|\xi|^2 \ dV_{\partial \Sigma} < \int_\Sigma H_\Sigma^2|\xi|\ dV_\Sigma.
\end{aligned}
\end{equation*}
Then
\[\Index(\Sigma)\geq\frac{2g+r-1-d}{2d}.\]
\end{corollary}

\begin{remark}
As mentioned before, all of the results above have a corresponding statement when $\Sigma$ is one-sided.
The proof follows the same arguments by considering instead the two-sided double cover $\tilde\Sigma$ and $1$-forms that are anti-invariant with respect to changing sheets.
If we denote by $\tilde{g}$ the genus of $\tilde{\Sigma}$ and $r$ the number of ends of $\Sigma$, then the corresponding index lower bounds are $\frac{\tilde g}{2d}$ and $\frac{\tilde g + r- 1 -d}{2d}$ in the closed and free boundary case respectively.
\end{remark}

\section{Applications}
In this section we discuss some examples in which our index estimates for CMC surfaces apply.
On some special cases that can be canonically embedded in $\real^n$ we characterize its stable CMC surfaces for a range of values of the mean curvature.

Let us first discuss some applications in the case of closed CMC surfaces in a Riemannian manifold without boundary.

\begin{theorem}\label{closed cmc case}
Let $M$ be a Riemannian manifold of dimension $3$ immersed in $\real^d$ with second fundamental form $\sffM_M$ and scalar curvature $\scalar_M$.
If $\Sigma$ is a closed two-sided immersed CMC surface with mean curvature $H_\Sigma^2>\sup|\sffM_M|^2-\inf \scalar_M$ then it satisfies
\begin{equation*}
\Index(\Sigma)\geq\frac{g}{d}.
\end{equation*}
In particular, if $\Sigma$ is stable then it must be a sphere.
\end{theorem}
\begin{proof}
Whenever $\xi\neq 0$ we may pick $e_1=\frac{\xi}{|\xi|}, e_2=\frac{\star\xi}{|\star\xi|}$ as an orthonormal basis.
Recalling that $|\xi|=|\star\xi|$, we have
\begin{equation*}
\sum_{i=1}^2|\sffM_M(e_i,\xi)|^2+|\sffM_M(e_i,\star\xi)|^2 = \sum_{i,j=1,2}|\sffM_M(e_i,e_j)|^2|\xi|^2
\end{equation*}
at every point.
Since $\sum_{i,j=1,2}|\sffM_M(e_i,e_j)|^2\leq |\sffM_M|^2$ we obtain
\begin{equation*}
\int_{\Sigma}\sum_{i=1}^2\left(|\sffM_M(e_i,\xi)|^2+|\sffM_M(e_i,\star\xi)|^2\right)-\scalar_{M}|\xi|^2 \ dV_\Sigma \leq\left(\sup|\sffM_M|^2-\inf \scalar_M\right)\int_\Sigma|\xi|^2 dV_\Sigma.
\end{equation*}
As long as $H_\Sigma^2>\sup|\sffM_M|^2-\inf \scalar_M$ the condition of Corollary \ref{index estimates} is satisfied.
\end{proof}

The following is a direct application of Nash's embedding theorem and the Theorem above.

\begin{corollary}
Let $M$ be a closed Riemannian manifold of dimension $3$. 
There exist constants $C>0$ and $H_0\geq 0$ depending on $M$ such that every closed two-sided immersed CMC surface $\Sigma$ of genus $g$ in $M$ and mean curvature $|H_\Sigma|>H_0$ satisfy
\begin{equation*}
\Index(\Sigma)\geq Cg
\end{equation*}
In particular, if $\Sigma$ is stable then it must be a sphere.
\end{corollary}

\begin{remark}
The result above proves the CMC version of Schoen and Marques-Neves' conjecture for minimal surfaces for sufficiently large mean curvature in arbitrary $3$-manifolds.
\end{remark}

First we would like to mention that Barbosa-do Carmo-Eschenburg \cite{barbosa-docarmo-eschenburg1988} have studied CMC hypersurfaces in simply connected spaces of constant sectional curvature and proved that geodesic spheres are the only stable ones.
In \cite{cavalcante-oliveira2017:arxiv} Cavalcante-de Oliveira have proved this using similar index estimates.
Let us now present some examples that satisfy the conditions of the above results using canonical embeddings.

\hfill

\subsection*{Closed CMC surfaces in $S^2\times \real$.}
\hfill

In \cite{souam2010} Souam proved that every stable CMC in $S^2\times \real$ is a rotational sphere.
We are going to give an alternative proof using our index estimates.

\begin{theorem}[{\cite[4.1]{souam2010}}]
Let $\Sigma$ be a closed two-sided immersed CMC surface in $S^2(r)\times\real$.
Then,
\begin{equation*}
\Index(\Sigma)\geq\frac{g}{4}.
\end{equation*}
In particular, if in addition $\Sigma$ is stable, then it is a sphere.
\end{theorem}
\begin{proof}
Consider the canonical embedding of $S^2(r)\times\real $ in $\real^4$.
Its second fundamental form is given by $\sffM_{S^2(r)\times\real}(X,Y)=\frac{1}{r}\langle {\pi}_*X,{\pi}_*Y\rangle$, where $\pi$ denotes the projection onto $S^2(r)$.
In particular, $|\sffM_{S^2(r)\times\real}|^2 =\frac{2}{r^2}$.
Since $\scalar_{S^2(r)\times \real}=\frac{2}{r^2}$, it follows from Theorem \ref{closed cmc case} that $\Index(\Sigma)\geq\frac{g}{4}$ whenever $H_\Sigma\neq 0$.
It remains to consider the case $H_\Sigma=0$.
However, by moving spherical slices, it follows from the maximum principle that the only closed minimal surfaces in are slices $S^2(r)\times\{t\}$ thus finishing the proof.
\end{proof}

\subsection*{Closed CMC surfaces in $T^2(\alpha,\beta)\times\real$.}
\hfill

Let $T^2(\alpha,\beta)$ denote the $2$-torus defined by the quotient of $\real^2$ with respect to the lattice $\Gamma(\alpha,\beta)$ generated by $\{(0,1),(\alpha,\beta)\}$ satisfying $0\leq\alpha\leq\frac{1}{2}$, $\beta>0$ and $\alpha^2+\beta^2\geq 1$.
We remark that any flat torus is isometric, up to dilations, to the above.

Let us consider the following map $f:\real^2\rightarrow \real^6$, $f=(f_1,f_2,f_3,f_4,f_5,f_6)$ defined by
\begin{equation*}
   \begin{aligned}
      f_1 & = C_1\cos(a_1 u + b_1 v) \\
      f_2 & = C_1\sin(a_1 u + b_1 v) \\
      f_3 & = C_2\cos(a_2 u + b_2 v) \\
      f_4 & = C_2\sin(a_2 u + b_2 v) \\
      f_5 & = C_3\cos(b_3 v) \\
      f_6 & = C_3\sin(b_3 v).
   \end{aligned}
\end{equation*}
If we choose $a_i=2\pi k_i$, $b_i=\frac{2\pi}{\sqrt{\alpha^2+\beta^2}}\left(l_i-k_i\frac{\alpha}{\beta}\sqrt{\alpha^2+\beta^2}\right)$ for positive integers $k_i,l_i$, i=1,2 and $b_3=\frac{2\pi}{\sqrt{\alpha^2+\beta^2}}$ then $f$ becomes an embedding of $T^2(\alpha,\beta)$ into $\real^6$.
Furthermore, as long as the $l_i,k_i$ satisfy
\begin{equation*}
   \begin{aligned}
      \frac{\alpha}{\beta}\sqrt{\alpha^2+\beta^2} < & \frac{l_1}{k_1} < \left(\frac{\alpha}{\beta}+1\right)\sqrt{\alpha^2+\beta^2}\\
      \left(\frac{\alpha}{\beta}-1\right)\sqrt{\alpha^2+\beta^2} < & \frac{l_1}{k_1} < \frac{\alpha}{\beta}\sqrt{\alpha^2+\beta^2},
   \end{aligned}
\end{equation*}
it is always possible to find positive constants $C_1,C_2,C_3$ so that $f$ is an isometric embedding (see \cite{kantor-frangulov1984}).

\begin{theorem}
Let $k_i,l_i$, $i=1,2$ be positive integers and $C_1,C_2,C_3>0$ so that $f$ defined above is an isometric embedding of $T^2(\alpha,\beta)$ into $\real^6$.
If $\Sigma$ is closed two-sided immersed CMC surface of genus $g$ in $T^2(\alpha,\beta)\times\real$ with mean curvature $H_\Sigma^2>C_1^2(a_1^2+b_1^2)^2+C_2^2(a_2^2+b_2^2)^2 + C_3^2b_3^4$ then
\begin{equation*}
\Index(\Sigma)\geq\frac{g}{7}.
\end{equation*}
In particular, if in addition $\Sigma$ is stable, then it is a sphere.
\end{theorem}
\begin{proof}
We embed $T^2(\alpha,\beta)\times\real$ in $\real^7$ using $f\times id$ and choose the following normal frame:
\begin{equation*}
   \begin{aligned}
      N_1 & = (-\cos(a_1 u+ b_1 v), -\sin(a_1 u + b_1 v), 0, 0, 0, 0, 0)\\
      N_2 & = (0, 0, -\cos(a_2 u+ b_2 v), -\sin(a_2 u + b_2 v), 0, 0, 0)\\
      N_3 & = (0, 0, 0, 0, -\cos(b_3 v), -\sin(b_3 v), 0)\\
      N_4 & = M \bigg{(} a_2C_2\sin(a_1u+b_1v),-a_2C_2\cos(a_1u+b_1v), -a_1C_1\sin(a_2u+b_2v), \\ 
          & a_1C_1\cos(a_2u+b_2v),
                \left.  -\frac{(a_2b_1-a_1b_2)C_1C_2}{b_3C_3}\sin(b_3 v), \frac{(a_2b_1-a_1b_2)C_1C_2}{b_3C_3}\cos(b_3 v), 0\right)
   \end{aligned}
\end{equation*}
with $M = \frac{1}{a_1^2C_1^2+a_2^2C_2^2+\frac{C_1^2C_2^2}{b_3^2C_3^2}(a_2b_1-a_1b_2)^2}$.

The shape operators for each normal direction and basis $\frac{\partial f}{\partial u}, \frac{\partial f}{\partial v}$ are given by
\begin{equation*}
   S_1= C_1\left(
   \begin{aligned}
   a_1^2 \quad a_1b_1\\
   a_1b_1 \quad b_1^2
   \end{aligned}\right),
   S_2= C_2\left(
   \begin{aligned}
   a_2^2 \quad a_2b_2\\
   a_2b_2 \quad b_2^2
   \end{aligned}\right),
   S_3= C_3\left(
   \begin{aligned}
   0 \quad\, 0\\
   0 \quad b_3^2
   \end{aligned}\right),
   S_4 = 0.   
\end{equation*}
So, we can compute $|S_1|^2 = C_1^2(a_1^2+b_1^2)^2$, $|S_2|^2 = C_2^2(a_2^2+b_2^2)^2$ and $|S_3|^2=C_3^2b_3^4$.
From which follows that
\begin{equation*}
|\sffM_{T^2\times\real}|^2 = C_1^2(a_1^2+b_1^2)^2 + C_2^2(a_2^2+b_2^2)^2 + C_3^2b_3^4.
\end{equation*}
Since $\scalar_{T^2\times\real}=0$, the result follows from Theorem \ref{closed cmc case}.
\end{proof}

In \cite{hauswirth-perez-Romon-ros2004} Hauswirth-Perez-Romon-Ros study doubly periodic isoperimetric surfaces and conjecture that the isoperimetric profile of $T^2(\alpha,\beta)\times \real$ is given by spheres, cylinders around lines and pairs of planes.
The authors then prove the conjecture for sufficiently large values of $\beta$.
The result above provides support for the spherical part of the isoperimetric profile conjecture for any given values of $\alpha,\beta$.
It also says that any genus $1$ isoperimetric surface, hence stable CMC, has a computable upper bound on the mean curvature, thus placing it further in the isoperimetric profile.

We would like to highlight two special cases.
Firstly, the hexagonal torus $T^2(\frac{1}{2},\frac{\sqrt{3}}{2})$ which is a critical case in Hauswirth-Perez-Romon-Ros' proof.
In this case we may choose, for example, $a_1=b_3=2\pi$, $a_2=4\pi$, $b_1=2\pi(1-\frac{1}{\sqrt{3}})$ and $b_2=2\pi(1-\frac{2}{\sqrt{3}})$.
Then, to make $f$ an isometric embedding we solve an underdetermined linear system so we may pick $C_1^2=1$, $C_2^2=\frac{1}{2}(\sqrt{3}-1)$ and $C_3^2=\frac{11+2\sqrt{3}}{6}$.
However, these are not unique choices and are far from being optimal in the sense that the lower bound for $H_\Sigma$ on the result above is minimal.

Secondly, the case of $\alpha=0$, that is, rectangular tori.
Rectangular tori can be embedded in $\real^4$ canonically and the estimate on $H_\Sigma$ is better.

\begin{corollary}
Let $T^2=T^2(0,\beta)$, $\beta\geq 1$, be a rectangular torus.
If $\Sigma$ is a closed two-sided immersed CMC surface of genus $g$ in $T^2\times \real$ with mean curvature $H_\Sigma^2>(2\pi)^2\left(1+\frac{1}{\beta^2}\right)$, then
\begin{equation*}
\Index(\Sigma)\geq \frac{g}{5}.
\end{equation*}
In particular, if in addition $\Sigma$ is stable, then it is a sphere.
\end{corollary}
\begin{proof}
In the case of a rectangular tori the embedding reduces to
\begin{equation*}
   \begin{aligned}
      f_1 & = \frac{1}{2\pi}\cos(2\pi u) \\
      f_2 & = \frac{1}{2\pi}\sin(2\pi u) \\
      f_3 & = \frac{\beta}{2\pi}\cos(\frac{2\pi}{\beta} v) \\
      f_4 & = \frac{\beta}{2\pi}\sin(\frac{2\pi}{\beta} v).
   \end{aligned}
\end{equation*}
We then embed $T^2\times\real$ in $\real^5$ and using the canonical normal frame the shape operators satisfy $|S_1|^2=(2\pi)^2$ and $|S_2|^2=\left(\frac{2\pi}{\beta}\right)^2$.
From which follows that
\begin{equation*}
|\sffM_{T^2\times\real}|^2 = (2\pi)^2\left(1+\frac{1}{\beta^2}\right).
\end{equation*}
Since the scalar curvature is zero, the result follows from Theorem \ref{closed cmc case}.
\end{proof}

\subsection*{Closed CMC surfaces in $T^3$.}
\hfill

Let us discuss the case of triply periodic CMC surfaces in $\real^3$.
Similarly to the doubly periodic case, Hauswirth-Perez-Romon-Ros conjecture that the isoperimetric profile on $T^3$ is given only by spheres, cylinders around lines and pairs of planes.
The following result supports the conjecture on the spherical branch in the same manner as above.
Any rectangular tori is equivalent, up to dilations, to the quotient of $\real^3$ by the lattice generated by $\{(1,0,0), (0,A,0), (0,0,B)\}$.
We identify such tori with $S^1(1)\times S^1(r_1)\times S^1(r_2)$ with $r_1=\frac{2\pi}{A}$ and $r_2=\frac{2\pi}{B}$.

\begin{theorem}
Let $T^3=S^1(1)\times S^1(r_1)\times S^1(r_2)$ be a rectangular $3$-torus.
If $\Sigma$ is a closed two-sided immersed CMC surface of genus $g$ in $T^3$ with mean curvature $H_\Sigma^2>1+\frac{1}{r_1^2}+\frac{1}{r_2^2}$, then 
\begin{equation*}
\Index(\Sigma)\geq \frac{g}{6}.
\end{equation*}
In particular, if in addition $\Sigma$ is stable, then it is a sphere.
\end{theorem}
\begin{proof}
Consider the canonical isometric embedding of $T^3$ into $\real^6$.
Using the usual normal frame the corresponding shape operators satisfy $|S_1|^2=1$, $|S_2|^2=\frac{1}{r_1^2}$ and $|S_3|^2=\frac{1}{r_2^2}$.
Thus,
\begin{equation*}
|\sffM_M|^2=1+\frac{1}{r_1^2}+\frac{1}{r_2^2}.
\end{equation*}
Since the ambient space is flat, the statement follows from Theorem \ref{closed cmc case}.
\end{proof}

\subsection*{Closed CMC surfaces in Berger Spheres $S_b^3(\kappa,\tau)$}
\hfill

Let $S^3\subset \complex^2$ be the unit sphere.
We denote by $S_b^3(\kappa,\tau)$ the Berger Sphere with metric $g_{\kappa,\tau}(X,Y)=\frac{4}{\kappa}\left(g(X,Y+\left(\frac{4\tau^2}{\kappa}-1\right)g(X,V)g(Y,V))\right)$, where $g$ is the Euclidean metric and $V=(iz,iw)$ is a Killing vector field on $S^3_b(\kappa,\tau)$.

Following Torralbo-Urbano \cite[Section 2.1]{torralbo-urbano2012}, as long as $\kappa-4\tau^2>0$ the authors construct an isometric embedding of $S^3_b(\kappa,\tau)$ into the complex projective space $\complex P^2(\kappa-4\tau^2)$ as a geodesic sphere. 
The space $\complex P^2(\kappa-4\tau^2)$ is the usual complex projective space endowed with the a dilation of the Fubini-Study metric $\frac{4}{\kappa-4\tau^2}g_{FS}$ with constant holomorphic sectional curvature $\kappa-4\tau^2$.

They then follow to embed $\complex P^2(\kappa-4\tau^2)$ into $\real^8$ and obtain a stability criterion similar to our index estimates in the case when the CMC surface is stable (see \cite[Lemma 6.4]{torralbo-urbano2012}).
By using the calculations already carried out in their paper we are able to improve their stability criterion in order to obtain the following index estimates.

\begin{theorem}
Let $S^3_b(\kappa,\tau)$ be a Berger Sphere with $\kappa-4\tau^2>0$.
If $\Sigma$ is a closed two-sided immersed CMC surface of genus $g$ in $S^3_b(\kappa,\tau)$ with mean curvature $H_\Sigma^2>\tau^2\left(\frac{\kappa}{4\tau^2}-3\right)\left(\frac{\kappa}{4\tau^2}+1\right)$, then
\begin{equation*}
\Index(\Sigma)\geq \frac{g}{8}.
\end{equation*}
In particular, if $\frac{\kappa}{4\tau^2}\in(1,3)$ then the index estimate is valid for all values of $H_\Sigma$.
\end{theorem}
\begin{proof}
For short we write $S^3_b=S^3_b(\kappa,\tau)$.
Consider the isometric embeddings $S^3_b\subset\complex P^2(\kappa-4\tau^2)\subset \real^8$ constructed in \cite[Section 2.1]{torralbo-urbano2012} and denote by $\sffM_{S^3_b}$ second fundamental form of $S^3_b$ in $\real^8$.
In the proof of \cite[Theorem 6.5:Case(4)]{torralbo-urbano2012} Torralbo-Urbano compute the following:
\begin{equation*}
\sum_{i,j=1}^2|\sffM_{S^3_b}(e_i,e_j)|^2= -6\tau^2+2\kappa + \tau^2\left(\frac{\kappa}{4\tau^2}-1\right)^2(1-C^2)^2,
\end{equation*}
where $e_1,e_2$ is an orthonormal frame on $\Sigma$ and $C=\frac{\kappa}{4\tau}g_{\kappa,\tau}(N,V)$ satisfies $0\leq C^2\leq 1$.

Since $\scalar_{S^3_b}=2(\kappa-\tau^2)$, it follows from the proof of Theorem \ref{closed cmc case} that the index estimate is valid as long as $H_\Sigma^2>\tau^2\left(\frac{\kappa}{4\tau^2}-1\right)^2-4\tau^2$, which concludes the proof.
\end{proof}

\subsection*{Closed CMC surfaces in pinched manifolds.}
\hfill

Similar to \cite{ambrozio-carlotto-sharp2018.1} we would like to mention that the following examples remain valid in the CMC case.

\begin{theorem}
Let $M$ be an immersed manifold in $\real^d$ with $\scalar_M>C|\vec{H}_M|^2$ for some $C>0$.
If $\Sigma$ is a closed two-sided immersed CMC surface of genus $g$ in $M$ with mean curvature $H^2_\Sigma\geq|\vec{H}_M|^2(1-2C)$, then
\begin{equation*}
\Index(\Sigma)\geq \frac{g}{d}.
\end{equation*}
In particular, if $\scalar_M>\frac{1}{2}|\vec{H}_M|^2$ then any CMC surface satisfy the index estimate above.
\end{theorem}
\begin{proof}
Observe that Gauss' equation on $M$ implies that $|\sffM_M|^2=|\vec{H}_M|^2-\scalar_M$.
As a consequence, we have 
\begin{equation*}
\sup|\sffM_M|^2-\inf \scalar_M<|\vec{H}_M|^2(1-2C).
\end{equation*}
The result follows from Theorem \ref{closed cmc case}.
\end{proof}

In the case of closed hypersurfaces of $\real^4$ we have the following weaker pinching result.
Compare to \cite[Remark 5.3]{ambrozio-carlotto-sharp2018.1}.

\begin{theorem}
Let $M$ be an embedded convex two-sided hypersurface of $\real^4$ with principal curvatures $0\geq k_1\geq k_2\geq k_3$ with respect to the outward normal vector field.
Suppose $M$ satisfies a pinching condition $\frac{k_3}{k_1}<C$, for some $C>0$.
If $\Sigma$ is a closed two-sided immersed CMC surface of genus $g$ in $M$ with mean curvature $H^2_\Sigma\geq 3k_1^2(C^2-2)$, then
\begin{equation*}
\Index(\Sigma)\geq \frac{g}{4}.
\end{equation*}
In particular, if $\frac{k_3}{k_1}<\sqrt{2}$ then any CMC surface satisfy the index estimate above.
\end{theorem}
\begin{proof}
Firstly observe that $|\sffM_M|^2=k_1^2+k_2^2+k_3^2\leq 3k_3^2$.
Secondly, it follows from Gauss' equation that $\scalar_M=2(k_1k_2+k_1k_3+k_2k_3)\geq 6k_1^2$.
That is,
\begin{equation*}
\sup |\sffM_M|^2-\inf \scalar_M\leq 3k_1^2\left(\frac{k_3^2}{k_1^2}-2\right)<3k_1^2(C^2-2).
\end{equation*}
The result follows from Theorem \ref{closed cmc case}.
\end{proof}

\subsection*{Free boundary CMC surfaces}
\hfill

To avoid repetition, let us mention that a similar index estimate is valid for free boundary CMC surfaces in mean convex domain of any of the above examples.
More especifically, let $\Sigma$ be a free boundary CMC surface of genus $g$ and $r$ boundary components in a domain with mean convex boundary, with respect to the inward normal direction, in any of the examples above.
Suppose $\Sigma$ satisfy the same mean curvature condition as the closed case in its respective ambient manifold.
Then $\Index(\Sigma)\geq\frac{2g+r-1-d}{2d}$ where $d$ is the dimension of the Euclidean space where the respective ambient manifold is embedded.

Let us highlight the special case of mean convex domains in $\real^3$.
The proof follows directly from Corollary \ref{index estimate free boundary}.
This particular case was first proved by Cavalcante-de Oliveira \cite{cavalcante-oliveira2018:arxiv}.

\begin{theorem}[{\cite[Theorem 1.1]{cavalcante-oliveira2018:arxiv}}]
Let $(M,\partial M)$ be a mean convex, with respect to the inward normal vector, domain of $\real^3$.
If $\Sigma$ is a free boundary CMC surface of $M$ of genus $g$ and $r$ boundary components, then
\begin{equation*}
\Index(\Sigma)\geq\frac{2g+r-4}{6}.
\end{equation*}
\end{theorem}

\bibliographystyle{plain}
\bibliography{bibliography}

\end{document}